\documentclass[12pt,leqno]{article}
\usepackage{latexsym,amsfonts,amssymb,amsmath,amscd}
\usepackage{amsthm}
\usepackage{pst-node}
\usepackage{pstricks}
\usepackage{graphicx}
\usepackage{xy}
\usepackage{lettrine}
\usepackage[latin1]{inputenc}
\usepackage[english]{babel}
\usepackage{epsfig}
\usepackage{wasysym}
\usepackage{amsfonts}
\usepackage{multicol}
\usepackage{wasysym}
\usepackage{color}
\usepackage{dsfont}
\usepackage{mathrsfs}
\usepackage{enumitem}
\usepackage{subfig} 
\newcommand{\suchthat}{\;\ifnum\currentgrouptype=16 \middle\fi|\;}
\usepackage{epstopdf}

\newtheorem{thm}{Theorem}
\newtheorem{prop}{Proposition}

\newtheorem{lem}{Lemma}

\newcommand{\vect}{\mathrm{Vect}}

\def\liminf{\mathop{\underline{\hbox{lim}}}\limits} 
\def\limsup{\mathop{\overline{\hbox{lim}}}\limits}

\newcommand{\ba}{\begin{array}}
\newcommand{\ea}{\end{array}}

\usepackage[normalem]{ulem}

\begin{document}
\sloppy  

\begin{center} 
{\Large{\textbf{Generalization of the Kimeldorf-Wahba correspondence for \textit{constrained} interpolation}}}
\vspace{0.4cm}

 Xavier Bay$^{\dagger}$, Laurence Grammont$^{\ddagger }$ and Hassan Maatouk$^{\dagger\ddagger}$
 \end{center}
\vskip0.2cm
\begin{center} 
($\dagger$) Mines de Saint-\'Etienne, 158 Cours Fauriel, 42023 Saint-\'Etienne, France\\
($\ddagger$) Universit\'e de Lyon, Institut Camille Jordan, UMR 5208, 23 rue du Dr Paul Michelon, 
42023 Saint-\'Etienne Cedex 2, France \\
bay,maatouk@emse.fr \& laurence.grammont@univ-st-etienne.fr
\end{center}
\vskip2cm

\noindent \textbf{Abstract}\\
In this paper, we extend the correspondence between Bayes' estimation and optimal interpolation in a Reproducing Kernel Hilbert Space (RKHS) to the case of linear inequality constraints such as boundedness, monotonicity or convexity. In the unconstrained interpolation case, the mean of the posterior distribution of a Gaussian Process (GP) given data interpolation is known to be the optimal interpolation function minimizing the norm in the RKHS associated to the GP. In the \textit{constrained} case, we prove that the Maximum \textit{A Posteriori} (MAP) or Mode of the posterior distribution is the optimal \textit{constrained} interpolation function in the RKHS. So, the general correspondence is achieved with the MAP estimator and \textit{not} the mean of the posterior distribution. A numerical example is given to illustrate this last result. 
\bigskip

\noindent {\bf Keywords~: correspondence; interpolation; inequality constraints; Reproducing Kernel Hilbert Space; Gaussian process; Bayesian estimation; Maximum A Posteriori}  

\bigskip

\noindent
{\bf AMS Classification~:} 

\section{Introduction}\label{intro}
Consider a function $y$ defined on a nonempty set $X$ of $\mathbb{R}^d \ (d\geq 1)$. The curve-fitting problem is to estimate $y$ using \textit{a prior} information and a finite set of noise-free evaluations~: 
\begin{equation*}
y\left(x^{(i)}\right)=y_i, \qquad i=1,\ldots,n,
\end{equation*}
where $x^{(1)},\ldots,x^{(n)}$ are $n$ distinct points of $X$. As in \cite{kimeldorf1970correspondence}, the \textit{prior} information is summarized by a zero-mean Gaussian Process (GP) $\{Y(x)\}_{x\in X}$ with covariance function  
\begin{equation}\label{K}
K(x,x'):=\mathds{E}(Y(x)Y(x')),
\end{equation} 
where $\mathds{E}$ denotes expectation. In this case, the usual Bayesian estimator $\hat{y}$ of $y$ is the mean of the posterior distribution of the GP $\{Y(x)\}_{x\in X}$ given data~:
\begin{equation*}
\hat{y}(x):=\mathds{E}\left(Y(x)\suchthat Y\left(x^{(1)}\right)=y_1,\ldots,Y\left(x^{(n)}\right)=y_n\right).
\end{equation*}
From \cite{Rasmussen:2005:GPM:1162254}, we have the following explicit expression for $\hat{y}$~:
\begin{equation}\label{unconstrKrig}
\hat{y}(x)=\boldsymbol{k}(x)^\top\mathds{K}^{-1}\boldsymbol{y}, \qquad x\in X,
\end{equation}
where $\boldsymbol{k}(x)=\left(K\left(x,x^{(1)}\right),\ldots,K\left(x,x^{(n)}\right)\right)^\top$, $\mathds{K}$ is the matrix 
$\left(K\left(x^{(i)},x^{(j)}\right)\right)_{1\leq i,j\leq n}$ and $\boldsymbol{y}=(y_1,\ldots,y_n)^\top$. \

On the other hand, it is well known (see \cite{wahba1990spline}) that this estimation function \eqref{unconstrKrig} is the unique solution of the following optimization problem~:
\begin{equation}\label{OptInt}\tag{$Q$} 
\min_{h\in H\cap I}\|h\|_H^2,
\end{equation}
where $H$ is the Reproducing Kernel Hilbert Space (see \cite{aronszajn1950}) associated to the positive definite kernel $K$ defined by \eqref{K} and $I$ is  
the set of interpolant functions~:    
\begin{eqnarray}\label{I}
I:=\left\{f\in \mathbb{R}^X~: \ f\left(x^{(i)}\right)=y_i, \ i=1,\ldots,n\right\}.
\end{eqnarray}
This result will be referred to as the correspondence between Bayes' estimation and optimal interpolation in a RKHS or Kimeldorf-Wahba correspondence. \
  
Now, we suppose that the function $y$ is known to satisfy some properties or constraints such as boundedness, monotonicity or convexity. Formally, let $C$ be a closed convex set of $\mathbb{R}^X$ corresponding to such constraints. For instance, $C$ is of the form~: 
\begin{eqnarray*}
C=\left\{f\in\mathbb{R}^X~: \ -\infty\leq a\leq f(x)\leq b\leq +\infty, \ x\in X\right\} & \mbox{(boundedness)},
\end{eqnarray*}
\begin{eqnarray*}
C=\left\{f\in\mathbb{R}^X~: \ \forall x\leq x', \ f(x)\leq f(x')\right\} & \mbox{(monotonicity)},
\end{eqnarray*} 
\begin{eqnarray*}
C=\left\{f\in\mathbb{R}^X~: \ \forall \lambda\in [0,1], \ \forall x,x', \ f(\lambda x+(1-\lambda)x')\leq \lambda f(x)+(1-\lambda)f(x'))\right\} & (\mbox{convexity}).
\end{eqnarray*}
If $H \cap C\cap I \neq \varnothing$, the following convex optimization problem~:
\begin{equation}\label{hopt}\tag{$P$}
\min_{h\in H \cap C\cap I}\|h\|_H^2
\end{equation}
has a unique solution denoted by $h_{opt}$ (see e.g. \cite{doi:10.1137/0909048} and \cite{Utreras85}), which can be seen as the optimal \textit{constrained} interpolation function associated to the knots $x^{(i)}, \ i=1,\ldots,n$.   \
 
In the Bayesian framework, the problem is now to make inference from the conditional distribution of the GP $\{Y(x)\}_{x\in X}$ given $Y\in C$ (\textit{prior} information) and given data $Y\left(x^{(i)}\right)=y_i, \ i=1,\ldots,n$. This conditional distribution can be thought as a truncated multivariate normal distribution but in an infinite dimensional linear space.\

The aim of this paper is to prove that the \textit{constrained} interpolation function $h_{opt}$ solution of problem~\eqref{hopt} is the mode or Maximum \textit{A Posteriori} (MAP) of this posterior distribution $\{Y\suchthat Y\in C\cap I\}$. \

The paper is organized as follows~: in Section~\ref{naturalCorrespondence}, we consider the finite-dimensional case to get insight into the natural correspondence between \textit{constrained} interpolation functions and Bayes' estimators. Section~\ref{MainResult} is devoted to the main result. We approximate the original Gaussian process by a sequence of finite-dimensional Gaussian processes (see e.g. \cite{maatouk:hal-01096751}, \cite{quinonero2007approximation} and \cite{trecate1999finite}). The MAP estimator of the finite-dimensional approximation process is well defined. Furthermore, this sequence of MAP estimators is shown to be convergent to the optimal \textit{constrained} interpolation function solution of problem \eqref{hopt}. As a consequence, we can interpret $h_{opt}$ as the most likely function or mode of the posterior distribution $\{Y\suchthat Y\in C\cap I\}$. This result can be seen as a generalization of the Kimeldorf-Wahba correspondence in the case of curve-fitting (interpolation case) taking into account linear inequality constraints. This new correspondence is illustrated in Section~\ref{NumIll}.\\

\section{The natural correspondence for finite-dimensional Gaussian processes}\label{naturalCorrespondence}
In this section, we assume that $\{Y^m(x)\}_{x\in X}$ is a finite-dimensional or degenerate GP in the sense that~:
\begin{equation}\label{findim}
Y^m(x):=\sum_{j=1}^m\xi_j\phi_j(x), \qquad x\in X,
\end{equation}
where $\{\phi_j, \ 1\leq j\leq m\}$ is a set of $m$ linearly independent functions in $\mathbb{R}^X$ and $\boldsymbol{\xi}=\left(\xi_1,\ldots,\xi_m\right)^\top\in\mathbb{R}^{m}$ is a zero-mean Gaussian vector with covariance matrix $\Gamma_m$ assumed to be invertible. The covariance function of $Y^m$ can be expressed as
\begin{equation}\label{finK}
K_m(x,x')=\phi(x)^\top\Gamma_m\phi(x'),
\end{equation}
where $\phi(x)=\left(\phi_1(x),\ldots,\phi_m(x)\right)^\top$. Let
\begin{equation}\label{finH}
H_m:=\vect\left\{\phi_j, \ 1\leq j \leq m \right\}=\left\{h\in\mathbb{R}^X~: \ \exists (c_1,\ldots,c_m) \in \mathbb{R}^{m} , \ h=\sum_{j=1}^mc_j\phi_j\right\}
\end{equation}
be the linear space spanned by the basis functions $\phi_j$ and consider on $H_m$ the dot product $(h_1,h_2)_m=c_{h_1}^\top\Gamma_m^{-1}c_{h_2}$, where $c_{h_i}$ are the coordinates of $h_i$ with respect to the basis $\{\phi_1,\ldots,\phi_m\}$, $i=1,2$. Since $\Gamma_m\phi(x)$ is the vector of coordinates of $K_m(.,x) \in H_m$ (see equation \eqref{finK}), we have
\begin{equation*}
(h,K_m(.,x))_m=c_h^\top\Gamma_m^{-1}\Gamma_m\phi(x)=c_h^\top\phi(x)=h(x).
\end{equation*}
Hence, $(H_m,(.,.)_m)$ is the RKHS with reproducing kernel $K_m$. In the following proposition, we denote by $\overset{\circ}{\widehat{H_m\cap C}}$ the interior of $H_m \cap C$ in the finite-dimensional space $H_m$. 

\begin{prop}\label{propfin}
Let $\{Y^m(x)\}_{x\in X}$ be a process of the form \eqref{findim} and $H_m$ defined by \eqref{finH} be the RKHS associated with the kernel function $K_m$ given in \eqref{finK}. Let us assume that $C$ is a closed convex subset of $\mathbb{R}^X$ (for pointwise topology) and $\overset{\circ}{\widehat{H_m\cap C}}\cap I$ is nonempty, where $I:=\left\{f\in \mathbb{R}^X~: \ f\left(x^{(i)}\right)=y_i, \ i=1,\ldots,n\right\}$. \

Then, the MAP estimator $\hat{y}_m$ defined as the mode of the posterior distribution of  $\{Y^m\suchthat Y^m\in C\cap I\}$ is well defined and is equal to the \textit{constrained} interpolation function  $h_{opt,m}$ solution of 
\begin{equation*}
\arg\min_{h\in H_m\cap C\cap I}\|h\|_m^2.
\end{equation*} 
\end{prop}

\begin{proof}
Remark that the sample paths of $Y^m$ are in $H_m$ by definition. Hence, it makes sense to define the density of $Y^m$ with respect to the uniform reference measure $\lambda_m$ on $H_m$ (m-dimensional volume measure or Lebesgue measure). This density is defined up to a multiplicative constant and to give it an explicit expression, we consider the following linear isomorphism~:
\begin{equation*}\label{iso}
i~:c\in\mathbb{R}^{m} \longmapsto h:=\sum_{j=1}^mc_j\phi_j\in H_m.
\end{equation*}
We can define the measure $\lambda_m$ on $H_m$ as the image measure $\lambda_m:=i(dc)$, where $dc=dc_1\times\ldots \times dc_m$ is the m-dimensional volume measure in $\mathbb{R}^{m}$. So, if $B\in \mathcal{B}(H_m)$ is a Borelian subset of $H_m$, we have
\begin{equation*}
\lambda_m(B)=\int_{\mathbb{R}^m}\mathds{1}_{i^{-1}(B)}(c)dc_1\times\ldots\times dc_m.
\end{equation*} 
To calculate the probability density function (pdf) of $Y^m$, we write 
\begin{equation*}
\mathds{P}\left(Y^m\in B\right)=\mathds{P}\left(\boldsymbol{\xi}\in i^{-1}(B)\right).
\end{equation*} 
Using the fact that $\boldsymbol{\xi}$ is a zero-mean Gaussian vector $\mathcal{N}(0,\Gamma_m)$, we obtain
\begin{eqnarray*}
\mathds{P}\left(Y^m\in B\right)&=&\int_{\mathbb{R}^m}\mathds{1}_{i^{-1}(B)}(c)\frac{1}{\sqrt{2\pi}^m|\Gamma_m|^{1/2}}\exp{\left(-\frac{1}{2}c^\top\Gamma_m^{-1} c\right)}dc\\
&=&\int_{\mathbb{R}^m}\mathds{1}_{B}(i(c))\frac{1}{\sqrt{2\pi}^m|\Gamma_m|^{1/2}}\exp{\left(-\frac{1}{2}\|i(c)\|_m^2\right)}dc.
\end{eqnarray*}
By the transfer formula, we get 
\begin{equation*}
\mathds{P}\left(Y^m\in B\right)=\int_{H_m}\mathds{1}_B(h)\frac{1}{\sqrt{2\pi}^m|\Gamma_m|^{1/2}}\exp{\left(-\frac{1}{2}\|h\|_m^2\right)}d\lambda_m(h).
\end{equation*} 
Hence, the (unconstrained) density of $Y^m$ with respect to $\lambda_m$
is the function 
\begin{equation*}
h\in H_m \longmapsto \frac{1}{\sqrt{2\pi}^m|\Gamma_m|^{1/2}}\exp{\left(-\frac{1}{2}\|h\|_m^2\right)}.
\end{equation*}

Let us now introduce the inequality constraints described by the convex set $C$. In the Bayesian framework, the \textit{prior} is the following truncated pdf (with respect to $\lambda_m$)~:
\begin{equation*}
h\in H_m \longmapsto k^{-1}\mathds{1}_{\left(h\in H_m\cap C\right)}\exp{\left(-\frac{1}{2}\|h\|_m^2\right)},
\end{equation*} 
where $k\neq 0$ (since $\overset{\circ}{\widehat{H_m\cap C}}\neq \varnothing$) is a normalizing constant. Assume $\overset{\circ}{\widehat{H_m\cap C}}\cap I$ is nonempty, the posterior likelihood $L_{pos}$ defined as the pdf of $Y^m$ given data interpolation, is given by 
\begin{equation}\label{Lpos}
L_{pos}(h)=k^{-1}\mathds{1}_{(h\in H_m\cap C\cap I)}\exp{\left(-\frac{1}{2}\|h\|_m^2\right)},
\end{equation} 
where $k\neq 0$ (since $\overset{\circ}{\widehat{H_m\cap C}}\cap I\neq \varnothing$) is a different normalizing constant. Remark that this density $L_{pos}$ is defined with respect to the ($m-n$)-dimensional measure volume induced by $\lambda_m$ on the affine subspace $H_m\cap I$ of $H_m$. By definition, the MAP estimator $\hat{y}_m$ is the solution of the following optimization problem
\begin{equation*}
\arg\max L_{pos}(h)=\arg\min\left(-2 \log L_{pos}(h)\right).
\end{equation*}   
From expression \eqref{Lpos}, the MAP estimator $\hat{y}_m$ is the \textit{constrained} interpolation function $h_{opt,m}$ solution of 
\begin{equation*}
\arg\min_{h\in H_m\cap C\cap I}\|h\|_m^2.
\end{equation*}
\end{proof}

\section{The main result}\label{MainResult}
In a Bayesian statistical framework, the \textit{prior} is the probability distribution of a zero-mean GP $\{Y(x)\}_{x \in X}$ with covariance function $K$ defined by \eqref{K} and assumed to be definite. We suppose that the realizations of $Y$ are in the Banach space $E=\mathcal{C}^0(X)$, the set of continuous functions defined on a compact set $X$. For the sake of simplicity, we suppose that $X=[0,1]$. The results presented in this paper can be generalized to the multi-dimensional case. Let $H$ be the RKHS associated to the positive definite function $K$. Then, $H$ is an Hilbertian subspace of $E$ since
\begin{equation*}
\| h\|_E=\sup_{x\in X}|(h,K(.,x))_H|\leq c\|h\|_H,
\end{equation*}
where $c=\sup_{x\in X}K(x,x)^{1/2}<+\infty$ by continuity of the kernel function $K$. Here, we suppose that we have also \textit{a priori} information such as boundedness, monotonicity or convexity constraints. Assume that these properties are mathematically  described by the set $C$, where $C$ is a closed convex subset of $\mathbb{R}^X$ as in Section~\ref{naturalCorrespondence} (a fortiori, $C\cap E$ is also a closed convex set of $E$)\footnote{The application $f\in E \longrightarrow f\in \mathbb{R}^X$ is continuous.}. Finally, let $I$ be the set of data interpolating functions $I=\left\{f\in E~: \ f\left(x^{(i)}\right)=y_i, \ i=1,\ldots,n\right\}$. Our aim is to make inference from the posterior distribution of the Gaussian process $Y$, so we need to handle the conditional distribution 
\begin{equation*}
\left\{Y \suchthat Y \in C \ \mbox{and} \ Y\left(x^{(i)}\right)=y_i, \ i=1,\ldots,n\right\}.
\end{equation*}
 
\subsection{Approximation of the Gaussian process $Y$}   
Keeping in mind Section~\ref{naturalCorrespondence}, we approximate the GP $Y$ by the following finite-dimensional Gaussian process~: 
\begin{equation}\label{YN}
Y^N(x):=\sum_{j=0}^{N}Y(t_{N,j})\phi_{N,j}(x), \qquad x\in X,
\end{equation}
where $0=t_{N,0}\leq t_{N,1}\leq \ldots \leq t_{N,N}=1$ is a graded subdivision of $X=[0,1]$ such that $\delta_N=\max\{\vert t_{N,j+1}-t_{N,j}\vert, \ j=0,\ldots,N-1\} \underset{N\to +\infty} \longrightarrow 0$ and $\phi_{N,j}$ are the associated piecewise linear functions (or hat functions) such that $\phi_{N,j}(t_{N,i})=\delta_{ij}, \  0\leq i,j\leq N$, where $\delta_{ij}$ is the Kronecker's Delta function. Note that $\boldsymbol{\xi}:=\left(Y(t_{N,0}), \ldots,Y(t_{N,N})\right)^\top$ is a zero-mean Gaussian vector. By continuity of the sample paths of $Y$ and continuous piecewise linear approximation in the Banach space $E=\mathcal{C}([0,1])$, $Y^N$ converges uniformly to the original GP $Y$ when $N$ tends to infinity with probability one.\

To simplify the proof of the main result (see Theorem~\ref{CorrespThm} below), block matrix structures will be used. To get this structure, we rename the knots of the partition $\Delta_N=\{ t_0, \ldots, t_N\}$  such that 
 \begin{eqnarray}\label{gradedpartition}
\Delta_{N+1}=\Delta_{N} \cup \{ t_{N+1}\}.
\end{eqnarray}
The finite-dimensional approximation of Gaussian Processes (GPs) can be rewritten as 
\begin{equation*}\label{YNn}
Y^N(x):=\sum_{j=0}^{N}Y(t_j)\varphi_{N,j}(x),
\end{equation*}
where $\varphi_{N,j}$ is the hat function associated to the knot $t_j$. \

From Section~\ref{naturalCorrespondence}, $Y^N$ is a finite-dimensional GP with covariance function 
$$K_N(x,x')=\sum_{k,\ell=0}^NK(t_{k},t_{\ell})\varphi_{N,k}(x)\varphi_{N,\ell}(x')=\varphi(x)^\top\Gamma_N\varphi(x'),$$ where $\Gamma_N:=(K(t_k,t_{\ell}))_{0\leq k,\ell\leq N}$. Note that $\Gamma_N$ is invertible since $K$ is assumed to be definite. The corresponding RKHS is $H_N:=\vect\{\varphi_{N,j}, \ j=0,\ldots,N\}$ with the norm given by $\|h\|_{H_N}:=c_h^\top \Gamma_N^{-1} c_h$, where $c_h=\left(h(t_0),\ldots,h(t_N)\right)^\top$.\

Now, we can compute the posterior likelihood function and the mode (or MAP) estimator $\hat{y}_N$ as a function defined on $X$.  

\begin{prop}\label{prop14}
If $\overset{\circ}{\widehat{H_N\cap C}}\cap I \neq \varnothing$, the convex optimization problem
\begin{equation}\label{finiteOptProb}\tag{$P_N$}
\min_{h\in H_N\cap C\cap I}\|h\|_{H_N}^2
\end{equation} 
has a unique solution denoted by $h_{opt,N}$. Additionally, the posterior likelihood function of $Y^N$ incorporating inequality constraints and given data is of the form
\begin{equation}\label{finitePostDis}
L_{pos}^N(h)=k_N^{-1}\mathds{1}_{h\in H_N\cap C\cap I}\exp\left(-\frac{1}{2}\|h\|^2_{H_N}\right),
\end{equation}
where $k_N$ is a normalizing constant. Then, the MAP estimator $\hat{y}_N$ of the posterior distribution \eqref{finitePostDis} is the solution $h_{opt,N}$ of the problem \eqref{finiteOptProb}. 
\end{prop}
\begin{proof}
It is a consequence of Proposition~\ref{propfin} of Section~\ref{naturalCorrespondence}.
\end{proof}

According to the uniform convergence of $Y^N$ to $Y$, it is natural to define the MAP estimator $\hat{y}$ of the Gaussian process $Y$ as the limit, if it exists, of the MAP estimator $\hat{y}_N$ of $Y^N$ as N tends to infinity. 

\subsection{Asymptotic analysis}
This subsection is devoted to the main result of the paper. The aim is to prove that the limit $\hat{y}:=\lim\limits_{N\to +\infty}\hat{y}_N$ of the MAP estimator $\hat{y}_N$ of $Y^N$ exists in $E$ and is the optimal \textit{constrained} interpolation function $h_{opt}$ in $H$~:
\begin{equation*}
h_{opt}:=\arg\min_{h\in H\cap C\cap I}\|h\|_H^2,
\end{equation*}
where $H$ is the RKHS associated to the process $Y$, $C$ is the closed convex set of $\mathbb{R}^X$ describing the inequality constraints and $I$ is the set of interpolating functions. 
To reach this goal, we need to analyze the link between the nested linear subspaces $H_N$ in $E$ and the RKHS $H$ associated with the reproducing kernel $K$. To do this, we denote by $\pi_N$ the projection operator from $E$ onto $H_N$ defined by~:
\begin{equation*}
\forall f\in E, \qquad \pi_N (f):=\sum\limits_{j=0}^N f(t_{j})\varphi_{N,j}.
\end{equation*}

\begin{thm}\label{theorem1}
For any $f\in E$, let us define the sequence of real numbers $(m_N(f))_{N\geq 1}$ by
\begin{equation*}
m_N(f):=\|\pi_N(f)\|_{H_N}^2=c_f^\top \Gamma_N^{-1}c_f,
\end{equation*}
where $c_f:=\left(f(t_{0}),\ldots,f(t_{N})\right)^\top$. Then, $(m_N(f))_{N\geq 1}$ is nonnegative and increasing. Furthermore, the RKHS $H$ associated to the covariance function $K$ is characterized by  
\begin{equation*}\label{newdefH}
H=\left\{f\in E~: \  \sup_{N}m_N(f)<+\infty\right\}
\end{equation*}
and, for all $f\in H$,
\begin{equation}\label{normH}
\|f\|_H^2=\sup_Nm_N(f)=\lim_{N\to +\infty}m_N(f)=\lim_{N\to +\infty}\|\pi_N(f)\|_{H_N}^2. 
\end{equation}
In particular, for $f \in H$ and $N \geq 1$,
\begin{equation}\label{stabilitypiN}
 \|\pi_N(f)\|_{H_N}  \leq \|f\|_H.
\end{equation}
\end{thm} 

\begin{proof}
As $\Gamma_N$ is symmetric positive definite, the sequence $(m_N(f))_N$ is nonnegative. The indexing of the knots (see \eqref{gradedpartition}) leads to the following block structure~: 
\begin{equation*}
\Gamma_{N+1}:=\left(\begin{matrix}
\Gamma_N & \boldsymbol{a}\\
\boldsymbol{a}^\top & K(t_{N+1},t_{N+1})
\end{matrix}
\right), \qquad \mbox{where $a=(K(t_{0},t_{N+1})\ldots,K(t_{N},t_{N+1}))^\top$}.
\end{equation*}
The monotonicity property of the sequence $(m_N(f))_{N\geq 1}$ is a consequence of Lemma~\ref{inclemma} (see Section~\ref{TecLem}). Thus,  
\begin{equation*}
\lim_{N\to +\infty}m_N(f)=\sup_Nm_N(f)\in [0,+\infty]. 
\end{equation*} 
Let us prove that $H \subset\{f\in E~: \ \sup_{N}m_N(f)<+\infty\}$.
Let $f\in H$ and $f_N$ be the orthogonal projection of $f$ onto the space $\vect\left\{K(.,t_{i}), \ i=0,\ldots,N\right\}$ in $H$. Then
\begin{equation*}
\|f_N\|_{H}^2 \leq \|f \|_{H}^2.
\end{equation*}
According to the characterization of the orthogonal projection and the reproducing property in a RKHS, we have $f_N=\displaystyle\sum_{j=0}^N\beta_{N,j}K(.,t_{j})$, where $\beta_N=\left(\beta_{N,0},\ldots,\beta_{N,N}\right)^\top$ is the solution of $\Gamma_N\beta_N=c_f$.  Therefore,
$\beta_N=\Gamma_N^{-1}c_f$ and 
\begin{equation*}
\|f_N\|_{H}^2 = \beta_N^\top \Gamma_N \beta_N = c_f^\top \Gamma_N^{-1}c_f.
\end{equation*}
Hence, $\|f_N\|_{H}^2=m_N(f) \leq  \|f \|_{H}^2 < +\infty$ and $\sup_{N}m_N(f)<+\infty$. \

Let us prove now that $\{f\in E~: \ \sup_{N}m_N(f)<+\infty\} \subset H$. Let $f \in E$ be such that $\sup_{N}m_N(f)<+\infty$. Consider $f_N:=\sum_{j=0}^N\beta_{N,j}K(.,t_{j})$, where $\beta_{N}=\Gamma_N^{-1}c_f$. Then, $f_N\in H$ and $\|f_N\|_H^2=c_f^\top \Gamma_N^{-1}c_f\leq M<+\infty$. Thus, $(f_N)_N$ is a bounded sequence in the Hilbert space $H$. By weak compactness in $H$, it exists $(f_{N_k})_k$ such that $f_{N_k}\underset{k}\rightharpoonup f_{\infty}\in H$. In particular, for all $x \in [0,1]$, $f_{N_k}(x)=(f_{N_k},K(.,x))_H \underset{k}\longrightarrow f_{\infty}(x)$. But, for any fixed $j\geq 1$,  
\begin{equation*}
 f_{N_k}(t_j)=f(t_j), \qquad \text{for $k$ large enough}.
\end{equation*}
Hence, for all $j$, $f_{\infty}(t_j)=f(t_j)$ and $f=f_{\infty}\in H$ by continuity and density of the knots in $[0,1]$. This ends the proof of the first part of the characterization. \

To conclude, let $F$  be defined as $F:=\vect\left\{K(.,t_{j}), \ j \geq 0 \right\}$. If $g\in F^\perp$, we have $\left(g,K(.,t_{j})\right)_H=g(t_{j})=0,  \ j\geq 0$. Hence, by continuity, $g=0$ and $F^\perp=\{0\}$. So, by classical approximation in a Hilbert space, the orthogonal projection $f_N$  of any $f \in H$  onto the subspace $F_N:=\vect\left\{K(.,t_{j}), \ j=0,\ldots,N\right\}$ satisfies  
\begin{eqnarray*}
f_N \underset{N \to +\infty}\longrightarrow f \qquad \mbox{in $H$}.
\end{eqnarray*}
Therefore, $\|f_N\|_{H}^2 = m_N(f) \underset{N \to +\infty}\longrightarrow \|f\|_H^2$, which completes the proof of the theorem. \\
\end{proof}

Now, we can state the main result of the paper. \\
\begin{thm}[{\bf Correspondence between \textit{constrained} interpolation and Bayesian estimation}]\label{CorrespThm}
Under the following assumptions~:
{\em \begin{eqnarray*}
\mbox{(H1)} && \overset{\circ}{\widehat{H\cap C}}\cap I\neq \varnothing,\\ 
\mbox{(H2)} && \forall N, \ \pi_N(C)\subset C,
\end{eqnarray*}}
the convex optimization problem 
\begin{equation}\label{PN}\tag{$P_N$}
\min_{h\in H_N\cap C\cap I}\|h\|_{H_N}^2
\end{equation} 
has a unique solution denoted by $h_{opt,N}$ and 
{\em 
\begin{equation}\label{mainconv}
h_{opt,N}\underset{N\to +\infty}\longrightarrow h_{opt} \qquad \mbox{in $E=\mathcal{C}^0(X).$ }  
\end{equation}}
Furthermore, the MAP estimator $\hat{y}_N$ solution of 
\begin{equation*}
\arg\max_{h\in H_N}L_{pos}^N(h), 
\end{equation*}
where $L_{pos}^N(h)$ is defined in \eqref{finitePostDis}, coincides with $h_{opt,N}$ and we also have
{\em \begin{equation*}
\hat{y}_N\underset{N\to +\infty}\longrightarrow h_{opt} \qquad \mbox{in $E=\mathcal{C}^0(X).$ }
\end{equation*}}
\end{thm}

\begin{proof} 
To avoid some technical difficulties, we suppose that the data points belong to $\Delta_N$ for $N$ large enough~:
\begin{eqnarray*}
\mbox{(H0)} && \left\{x^{(i)}, \ i=1,\ldots,n\right\}\subset \Delta_N.
\end{eqnarray*}  
The proof without this last assumption can be found in \cite{bay:hal-01136466} and \cite{maatoukthesis2015}. \ 

Let $g\in H\cap C\cap I$, then $\pi_N(g)\in H_N$. As $\pi_N(C)\subset C$, $\pi_N(g)\in C$ and $\pi_N(g)\in I$ due to (H0). So, $H_N\cap C\cap I$ is a nonempty closed convex subset of $H_N$. Therefore, \eqref{PN} has an unique solution $h_{opt,N}$. Write
\begin{equation*}
\|h_{opt,N}-h_{opt}\|_E \leq \|h_{opt,N}-\pi_N(h_{opt})\|_E+\|\pi_N(h_{opt})-h_{opt}\|_E.
\end{equation*}
 We know from approximation theory in the Banach $E=\mathcal{C}^0(X)$ that 
\begin{equation*}
\|\pi_N(h_{opt})-h_{opt}\|_E \underset{N\to +\infty}\longrightarrow 0.
\end{equation*}
According to the Lemma~\ref{lemma1} of Section~\ref{TecLem},
\begin{eqnarray*}
\|h_{opt,N}-\pi_N(h_{opt})\|_{E}^2  \leq c^2 \|h_{opt,N}-\pi_N(h_{opt})\|_{H_N}^2.  
\end{eqnarray*}
Write now in $H_N$
\begin{equation}\label{numero}
\| h_{opt,N}-\pi_N(h_{opt})\|_{H_N}^2 =\|h_{opt,N}\|_{H_N}^2+\|\pi_N(h_{opt})\|_{H_N}^2-2\left(h_{opt,N},\pi_N(h_{opt})\right)_{H_N}.
\end{equation}
As $h_{opt,N}$ is the orthogonal projection of $0$ onto the convex set $H_N\cap C\cap I$ in the Hilbert space  $H_N$ and $\pi_N(h_{opt})\in H_N\cap C\cap I$, we have
\begin{equation*}
\left(0-h_{opt,N},\pi_N(h_{opt})-h_{opt,N}\right)_{H_N}\leq 0.
\end{equation*}
Therefore, 
\begin{equation*} 
\|h_{opt,N}-\pi_N(h_{opt})\|_{H_N}^2\leq \|\pi_N(h_{opt})\|_{H_N}^2-\|h_{opt,N}\|_{H_N}^2, 
\end{equation*}
so that, by \eqref{stabilitypiN}
\begin{equation}\label{inefin}
\|h_{opt,N}-\pi_N(h_{opt})\|_{H_N}^2\leq \| h_{opt}\|_{H}^2-\|h_{opt,N}\|_{H_N}^2.
\end{equation}
From \eqref{inefin}, it is sufficient to prove 
\begin{equation*}
\|h_{opt,N}\|_{H_N}^2=\min_{h\in H_N\cap C\cap I}\|h\|_{H_N}^2\underset{N\to +\infty}\longrightarrow \|h_{opt}\|_H^2=\min_{h\in H\cap C\cap I}\|h\|_H^2.
\end{equation*}
As $\pi_N(h_{opt}) \in H_N\cap C \cap I$ and by \eqref{stabilitypiN},
\begin{equation*}
\|h_{opt,N}\|_{H_N}^2\leq \|\pi_N(h_{opt})\|_{H_N}^2 \leq \| h_{opt} \|_{H}^2.
\end{equation*}
Hence, 
\begin{equation}\label{proplimsup}
\limsup_N\|h_{opt,N}\|_{H_N}^2\leq \|h_{opt}\|_{H}^2.
\end{equation}
Let $\tilde{h}_N$ be the solution of the problem 
\begin{equation*} \label{rhoN}
\min_{h\in H} \left\{\| h\|^2_{H}~: \ h(t_{j})=h_{opt,N}(t_{j}) , \ j=0,\ldots,N \right\}.
\end{equation*}
It can be expressed as    
\begin{equation*}\label{condGP}
\tilde{h}_N =\boldsymbol{k_N}(.)^\top \Gamma_N^{-1}c_{h_{opt,N}}, 
\end{equation*}
where $\boldsymbol{k_N}(.)=\left(K\left(.,t_0\right),\ldots,K\left(.,t_N\right)\right)^\top$.  Then, we get $\|\tilde{h}_N\|_H=c_{h_{opt,N}}^{\top}\Gamma_N^{-1}c_{h_{opt,N}}=\|h_{opt,N}\|_{H_N}$. By \eqref{proplimsup}, $(\|\tilde{h}_N\|_H)_N$ is a bounded sequence in $H$. By weak compactness, there exists a sub-sequence $\tilde{h}_{N_k}$ such that
\begin{equation}\label{weakcompact}
\tilde{h}_{N_k}\underset{k\to +\infty}\rightharpoonup h_{\infty}\in H, \qquad \text{(weak convergence)}.
\end{equation}
Let us prove that $h_{\infty} \in C$. For fixed $j$ and for $k$ large enough, $\tilde{h}_{N_k}(t_j)=h_{opt,N_k}(t_j)\underset{k\to +\infty}\longrightarrow h_{\infty}(t_j)$. Hence, $\pi_N(h_{opt,N_k})\underset{k\to +\infty}\longrightarrow \pi_N(h_{\infty})$ for any fixed  $N\geq 1.$  
As $H_N\cap C$ is closed in $H_N$ and  $\pi_N(h_{opt,N_k}) \in C$, we have  $\pi_N(h_{\infty})\in C.$ As $\pi_N(h_{\infty})\underset{N\to +\infty}\longrightarrow h_{\infty}$ in $E$ and $C$ is closed in $E=\mathcal{C}^0(X)$, we conclude that $h_{\infty}\in C$. \

Let us show now that $h_{\infty} \in I$. As $x^{(i)}\in \Delta_N$ for $N$ large enough, we get $\tilde{h}_{N_k}\left(x^{(i)}\right)=h_{opt,N}\left(x^{(i)}\right)=y_i$. As $\tilde{h}_{N_k}\left(x^{(i)}\right)=\left(\tilde{h}_{N_k},K\left(.,x^{(i)}\right)\right)_H$ and $\tilde{h}_{N_k}\underset{k\to +\infty}\rightharpoonup h_{\infty}$, we have $h_{\infty}\left(x^{(i)}\right)=y_i$. Hence $h_{\infty}\in I$. \

From property \eqref{weakcompact}, equality $\|\tilde{h}_N\|_H=\|h_{opt,N}\|_{H_N}$ and inequality \eqref{proplimsup}, we have
\begin{equation*}
\|h_{\infty}\|_H^2\leq \liminf_k\|\tilde{h}_{N_k}\|_H^2 \leq \limsup_k\|\tilde{h}_{N_k}\|_{H}^2\leq \|h_{opt}\|_H^2.
\end{equation*}
Since $h_{\infty}\in H\cap C\cap I$, we have also $\|h_{opt}\|_H^2\leq \|h_{\infty}\|_H^2$ so that $\|h_{opt}\|_H^2=\|h_{\infty}\|_H^2$ and thus $\lim_k\|\tilde{h}_{N_k}\|_{H}^2=\|h_{opt}\|_H^2$. Since norm convergence and weak convergence (see \eqref{weakcompact}) imply strong convergence, we have
\begin{equation*}\label{strongcompact}
\tilde{h}_{N_k}\underset{k\to +\infty}\longrightarrow h_{\infty}\in H,
\end{equation*}
and also $\tilde{h}_{N}\underset{N\to +\infty}\longrightarrow h_{\infty}\in H$ by a classical compacity argument.
Hence, $\lim\limits_{N}\|h_{opt,N}\|_{H_{N}}^2=\lim\limits_{N}\|\tilde{h}_{N}\|_{H}^2=\|h_{\infty}\|_{H}^2=\|h_{opt}\|^2_H$. Then from \eqref{inefin}, $\|h_{opt,N}-\pi_N(h_{opt})\|_{H_N}^2  \underset{N\to +\infty}\longrightarrow  0$ and   
$$\|h_{opt,N}-h_{opt}\|_E \underset{N\to +\infty}\longrightarrow  0.$$
The second part is a consequence of Proposition~\ref{prop14}. 
\end{proof}

\paragraph{Comments} 
Remark that assumption (H1) is not restrictive and assumption (H2) is ensured for applications in consideration in this paper (boundedness, monotonicity or convexity constraints). For instance, if $f$ is a non-decreasing function on $[0,1]$, then the piece-wise linear interpolation $\pi_N(f)$ is also non-decreasing for any $N$. For a general convex set $C$, the sequence of approximation $(\pi_N(f))_N$ must be adapted to satisfy assumption (H2). \

Now, the \textit{constrained} optimization problem has a nice probabilistic interpretation as a Bayesian estimator of a function $y\in \mathcal{C}^0(X)$. The function $h_{opt}=\hat{y}:=\lim_N\hat{y}_N$ can be thought as the most likely function in the subspace $C$ of \textit{constrained} functions $h$ satisfying $h\left(x^{(i)}\right)=y_i, \ i=1,\ldots,n$. Theorem~\ref{CorrespThm} proves that this estimator $\hat{y}$ is independent of the choice of the subdivision $\{t_{j}\}$ and is a smooth function since $\hat{y}=h_{opt}$ is the solution of a \textit{constrained} interpolation problem in a RKHS. 

\subsection{Technical lemmas}\label{TecLem}
\begin{lem}\label{inclemma}
Let $B:=\left(\begin{matrix}
A & \boldsymbol{a}\\
\boldsymbol{a}^\top & \alpha
\end{matrix}
\right)$ be a real block matrix where $A$ is an $N\times N$ matrix, $\boldsymbol{a}$ is an $N\times 1$ vector and $\alpha\in \mathbb{R}$. Assume that $B$ is symmetric positive definite. Let $\boldsymbol{y}=(\boldsymbol{x},y_{N+1})^\top$, where $\boldsymbol{x}$ is an $N\times 1$ vector and $y_{N+1}\in\mathbb{R}$. Then,
\begin{equation*}\label{matrixInq}
\boldsymbol{y}^\top B^{-1}\boldsymbol{y}\geq \boldsymbol{x}^\top A^{-1}\boldsymbol{x}. 
\end{equation*}
\end{lem}

\begin{proof}[Proof of Lemma~\ref{inclemma}]
Write $\boldsymbol{y}=B\boldsymbol{v}$ with $\boldsymbol{v}=B^{-1}\boldsymbol{y}=\left(\begin{matrix}
\boldsymbol{u}\\
v_{N+1}
\end{matrix} \right)$. By block matrix multiplication, we have
\begin{equation*}
\boldsymbol{x}=A\boldsymbol{u}+v_{N+1}\boldsymbol{a} \quad \text{and} \quad y_{N+1}=\boldsymbol{a}^\top \boldsymbol{u}+\alpha v_{N+1}.
\end{equation*}
Now, $\boldsymbol{y}^\top B^{-1}\boldsymbol{y}=\boldsymbol{v}^\top B\boldsymbol{v}=\boldsymbol{u}^\top A\boldsymbol{u}+2v_{N+1}\boldsymbol{a}^\top \boldsymbol{u}+\alpha v_{N+1}^2$ and $\boldsymbol{x}^\top A^{-1}\boldsymbol{x}=\boldsymbol{u}^\top A\boldsymbol{u}+2v_{N+1}\boldsymbol{a}^\top \boldsymbol{u}+v_{N+1}^2\boldsymbol{a}^\top A^{-1}\boldsymbol{a}$. Comparing expression $\boldsymbol{y}^\top B^{-1}\boldsymbol{y}$ and $\boldsymbol{x}^\top A^{-1}\boldsymbol{x}$, we only need to prove the inequality~: $\alpha\geq \boldsymbol{a}^\top A^{-1}\boldsymbol{a}$. For this, consider the block vector $\boldsymbol{z}=\left(\begin{matrix}
A^{-1}\boldsymbol{a}\\
-1
\end{matrix}\right)$. Since $B$ is positive, $\boldsymbol{z}^\top B\boldsymbol{z}=\boldsymbol{a}^\top A^{-1}\boldsymbol{a}-2\boldsymbol{a}^\top A^{-1}\boldsymbol{a}+\alpha=\alpha-\boldsymbol{a}^\top A^{-1}\boldsymbol{a}\geq 0$. \\ 
\end{proof}

\begin{lem}\label{lemma1}
For any $h\in H_N$, $\|h\|_E\leq c\|h\|_{H_N}$, where $c$ is a constant independent of $N$.  
\end{lem}
\begin{proof} 
For $x\in X$, we have
\begin{equation*}
|h(x)|=|(h,K_N(.,x))_{H_N}|\leq \|h\|_{H_N}\times \sqrt{K_N(x,x)},
\end{equation*}
where $K_N(x,x)=\sum_{i,j=0}^N K(u_{N,i},u_{N,j})\phi_{N,i}(x)\phi_{N,j}(x)$. Since $\sum_{i,j=0}^N\phi_{N,i}(x)\phi_{N,j}(x)=1$, we obtain
\begin{equation*}
0\leq \sup_{x\in X}K_N(x,x)\leq M=\max_{x,x'\in X}|K(x,x')|,
\end{equation*}
which completes the proof of the lemma.
\end{proof}

\section{Numerical illustration}\label{NumIll}
The aim of this section is to illustrate the correspondence established in previous sections between the MAP estimator and the \textit{constrained} interpolation function solution of problem \eqref{hopt}. We are interested in the case where the real function $f$ respects boundedness constraints. Thus, the convex set $C$ is equal to~:
\begin{equation*}
C=\left\{f\in\mathcal{C}^0\left([0,1]\right)~: \ -\infty\leq a\leq f(x)\leq b\leq +\infty, \ x\in [0,1]\right\}.
\end{equation*} 

\begin{figure}[hptb]
\begin{minipage}{.5\linewidth}
\centering
\subfloat[]{\label{boundedGP1}\includegraphics[scale=.4]{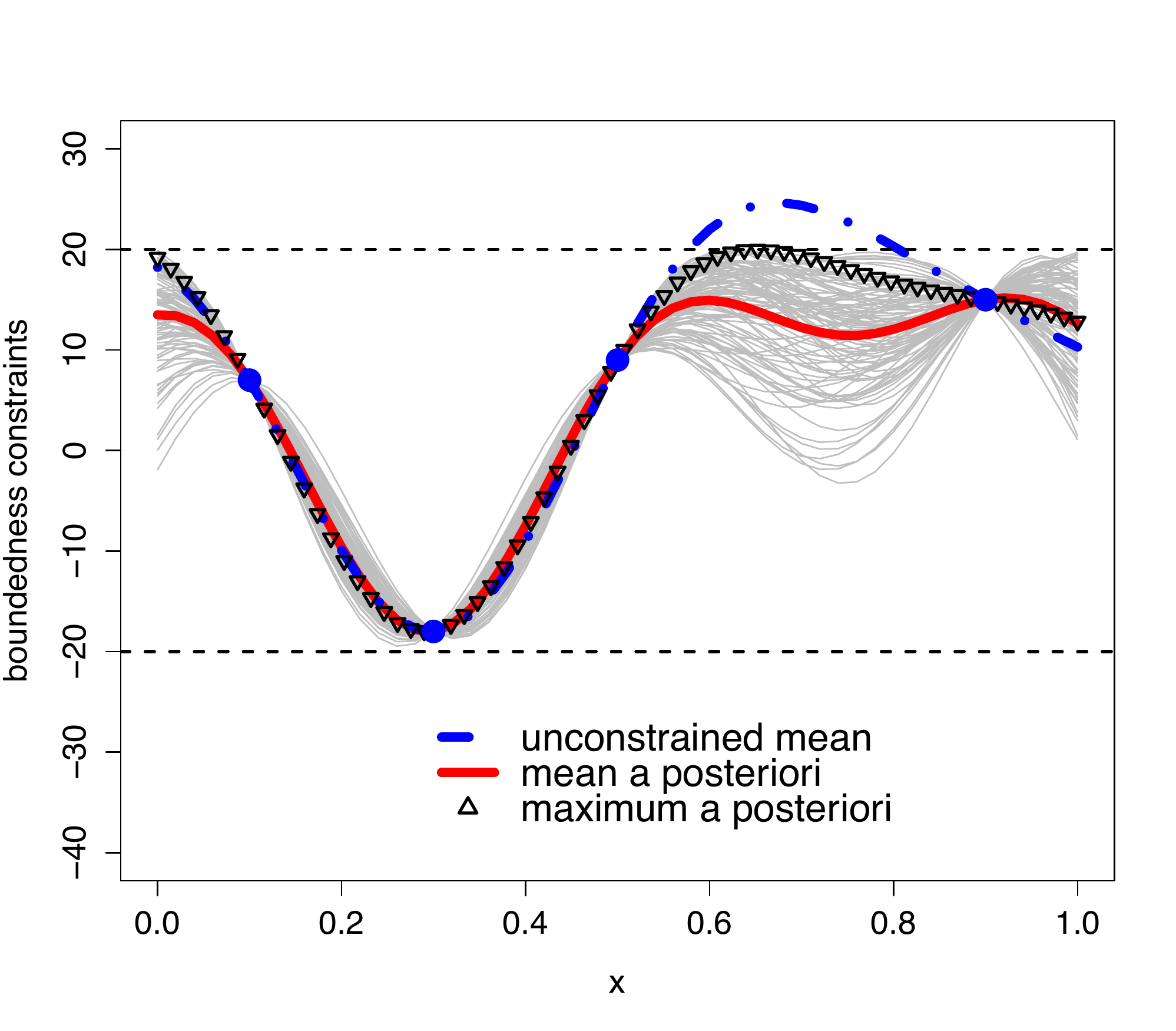}}
\end{minipage}%
\begin{minipage}{.5\linewidth}
\centering
\subfloat[]{\label{boundedGP2}\includegraphics[scale=.4]{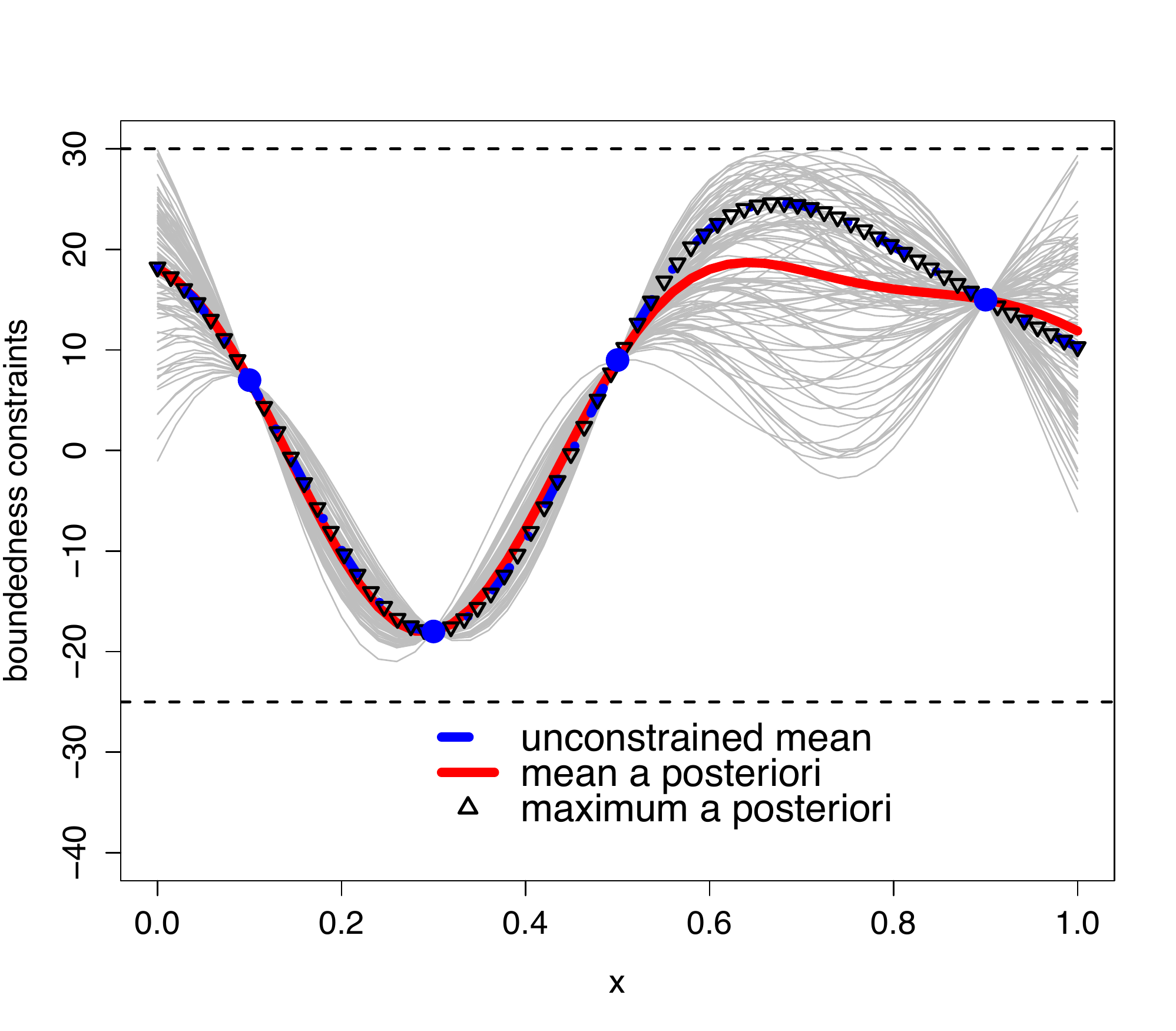}}
\end{minipage}
\caption{Unconstrained and constrained mean together with the maximum \textit{a posteriori} (MAP) estimator using the \textit{constrained} model. The lower and upper bounds are equal to $-25$ and $20$ (Figure~\ref{boundedGP1}) and equal to $-25$ and $30$ (Figure~\ref{boundedGP2}).}
\label{boundedGP}
\end{figure}

Now, we suppose that $f$ is evaluated at $n=4$ design points (see Figure~\ref{boundedGP}) with values in the interval $]-25,20[$ (Figure~\ref{boundedGP1}) and $]-25,30[$ (Figure~\ref{boundedGP2}). In both figures, the Gaussian covariance function is used which is defined as
\begin{equation*}
K(x,x'):=\sigma^2\exp\left(-\frac{(x-x')^2}{2\theta^2}\right),
\end{equation*}
where the hyper-parameters $(\sigma,\theta)$ are fixed to $(25,0.2)$. In Figure~\ref{boundedGP1}, we choose $N=50$ and generate 100 sample paths taken from the finite-dimensional approximation of Gaussian processes \eqref{YN} conditionally to interpolation conditions and boundedness constraints, where the lower and upper bounds are respectively -25 and 20 (the R package `constrKriging' is used in the simulation, see~\cite{maatoukpackage2015} for more details). Notice that the sample paths of the conditional Gaussian process (gray solid line) respect the boundedness constraints in the entire domain unlike the unconstrained mean \eqref{unconstrKrig}. In Figure~\ref{boundedGP2}, we just relax the boundedness constraints such that the unconstrained mean respects it. In that case, the unconstrained mean coincides with the MAP estimator but not with the mean of the simulation (i.e. posterior mean). Hence, in the constrained case, the mean of the posterior distribution does not correspond to the optimal interpolation function.

\begin{figure}
\centering
\includegraphics[scale=0.4]{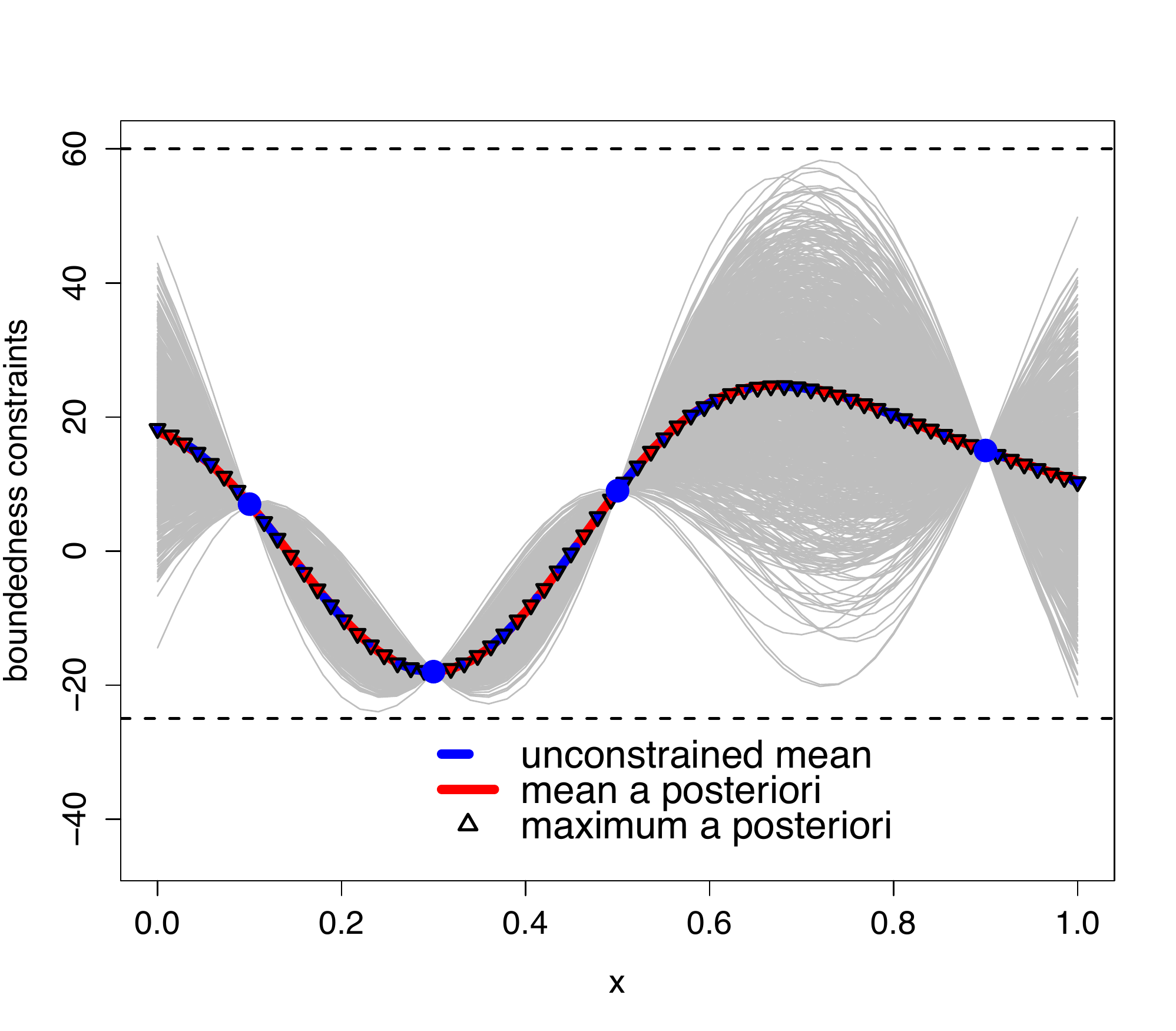}
\caption{1000 sample paths taken from the Gaussian process (gray solid line) respecting boundedness constraints between -25 and 60. The unconstrained mean, the mean and the maximum \textit{a posteriori} coincide.}
\label{unboundedGP}
\end{figure}

In Figure~\ref{unboundedGP}, we also relax the boundedness constraints such that they do not have an impact on the model. In that case, the unconstrained mean, the mean and the maximum of the posterior distribution coincide as expected.

\section{Conclusion}
In this paper, the correspondence between two approaches to solve an interpolation problem in the case of linear inequality constraints is established. On the first hand, a deterministic approach leads to solve a \textit{constrained} optimization problem under both interpolation conditions and inequality constraints in a Hilbert space. On the second hand, a probabilistic approach considers an estimation problem in a Bayesian framework. In the case of a finite-dimensional Gaussian process, the correspondence between the MAP estimator (maximum of the posterior distribution) and the \textit{constrained} interpolation function is proved. In the infinite-dimensional case, the correspondence is done by finite-dimensional approximation and convergence of the MAP estimator to the \textit{constrained} interpolation function. This result can be seen as a generalization of the correspondence established by Kimelford and Wahba in \cite{kimeldorf1970correspondence} between Bayesian estimation on stochastic process and curve fitting.

\section{Acknowledgements}
Part of this work has been conducted within the frame of the ReDice Consortium, gathering industrial (CEA, EDF, IFPEN, IRSN, Renault) and academic partners (\'Ecole des Mines de Saint-\'Etienne, INRIA, and the University of Bern) around advanced methods for Computer Experiments.

\bibliography{bibliographie}
\bibliographystyle{plain}
\end{document}